\documentclass{amsart}
\usepackage{amssymb, amsmath, amsthm, color, hyperref, url, fullpage,euscript,mathrsfs}

\usepackage[paper=a4paper,left=30mm,right=20mm,top=25mm,bottom=30mm]{geometry}
    \newenvironment{dedication}
        {\vspace{6ex}\begin{quotation}\begin{center}\begin{em}}
        {\par\end{em}\end{center}\end{quotation}}
\usepackage{ifxetex}
\ifxetex
  \usepackage{fontspec}
\else
  \usepackage[T1]{fontenc}
  \usepackage[utf8]{inputenc}
  \usepackage{lmodern}
\fi

\usepackage[all]{xy}
\usepackage[table]{xcolor}

\usepackage{epsfig}
\newtheorem{theorem}{Theorem}

\newtheorem{lemma}{Lemma}
\newtheorem{example}{Example}
\newtheorem{conj}{Conjecture}
\newtheorem{remark}{Remark}
\newtheorem{defn}{Definition}

\def\Q{\mathbb Q}
\def\C{\mathbb C}
\def\({\left(}
\def\){\right)}
\def\l{\lambda}

\def\G{\Gamma}
\def\a{\alpha}

\def \Z{\mathbb Z}
\def \F{\mathbb F}

\newcommand*\HYPERskip{&}
\catcode`,\active
\newcommand*\pFq{
\begingroup
\catcode`\,\active
\def ,{\HYPERskip}%
\doHyper
}
\catcode`\,12
\def\doHyper#1#2#3#4#5{%
\, _{#1}F_{#2}\left[\begin{matrix}#3 \smallskip \\  #4\end{matrix} \; ; \; #5\right]%
\endgroup
}

\title{Some Numeric Hypergeometric Supercongruences}

\author{Ling Long}

\address{303 Lockett Hall, Louisiana State University, Baton Rouge, LA 70803, USA, llong@lsu.edu}
\thanks{The work has been by NSF  DMS \#1602047. The numeric computation has been done by using both \texttt{Magma} and \texttt{Sagemath}. The author would like to thank Wadim Zudilin for his constant discussions, concrete and constructive suggestions, and Frits Beukers for his helpful correspondence. This article was motivated by the discussion held at the ``Hypergeometric motives and Calabi–Yau differential equations" workshop, which took place at the MATRIX institute, Melbourne Australia in January 2017. The author would also like to thank the referee for helpful sugguestions and comments leading to a refined version of this article.}
\date{}

\begin{document}
\maketitle

\begin{dedication}
{In celebration of Geoffrey Mason's 70th birthday.}
\end{dedication}


\begin{abstract}
In this article, we list a few hypergeometric supercongruence conjectures based on two evaluation formulas of Whipple and  numeric data computed using \texttt{Magma} and \texttt{Sagemath}.
\end{abstract}

\section{Introduction}

Let  $n$ be a positive integer, $a_1,\cdots, a_n,b_1, \cdots, b_n, \l\in \Q$, in particular $b_n=1$. Let $\alpha=\{a_1,\cdots,a_n\}$ and $\beta=\{b_1,\cdots,b_{n}\}$, be two multi-sets of rational numbers. Here $a_i$'s (resp. $b_j$'s) may repeat. The triple $\{\alpha,\beta;\l\}$ is called a \emph{hypergeometric datum}. The corresponding hypergeometric series (see \cite{AAR}) is defined by

\begin{equation*}
F(\alpha,\beta;\l):=\pFq{n}{n-1}{a_1&\cdots& a_{n-1}&a_n}{&b_1&\cdots&b_{n-1}}{\l}=\sum_{k=0}^\infty \frac{(a_1)_k\cdots(a_n)_k}{(b_1)_k\cdots (b_{n-1})_k}\frac{\l^k}{k!},
\end{equation*} where  $(a)_k=a(a+1)\cdots(a+k-1)=\frac{\G(a+k)}{\G(a)}$ with $\G(\cdot)$ being the Gamma function. As a function of $\l$, it satisfies an order-$n$ Fuchsian differential equation which has only three regular singularities located at $0,1,\infty$. The truncated hypergeometric series $F(\alpha,\beta;\l)_m:=\pFq{n}{n-1}{a_1&\cdots& a_{n-1}&a_n}{&b_1&\cdots&b_{n-1}}{\l}_m$ with subscript $m$ denotes the truncated series at $\l^m$th term.\medskip

Hypergeometric functions form an important class of special functions. They play many vital roles in differential equations, algebraic varieties, modular forms. In his recent papers \cite{CGM,CM}, Geoffrey Mason explored the role of hypergeometric functions in vector valued modular forms. In \cite{CGM} by  Franc, Gannon and Mason, $p$-adic properties of the coefficients of  hypergeometric functions are used to explain the unbounded denominator behavior of noncongruence vector valued modular forms. For related discussions on the  unbounded denominator behavior of genuine noncongruence modular forms, see \cite{Chen, KL,  LL}. In this paper we will discuss the $p$-adic properties of hypergeometric functions initiated in Dwork's work (see \cite{Dwork}).

\medskip

To present  Dwork type of congruence, we will first recall
\begin{defn}\label{def:dash}
Let $p$ be a fixed prime number and $r\in  \Z_p^\times$. Use $[r]_0$ to denote the first $p$-adic digit of $r$, namely an integer in $[0,p-1]$ which is congruence to $r$ modulo $p$. Dwork dash operation, a map from $\Z_p\rightarrow \Z_p$, is
defined by
$$
r'= (r+[-r]_0)/p,
$$
\end{defn} In other words,
$$r'p-r=[-r]_0.$$
For example, if $p$ is an odd prime, then $[-\frac12]_0=\frac{p-1}2$. If $r=\frac 76$ and $p\equiv 1\mod 3$, then $[-\frac 76]_0=\frac{p-7}6$. Hence $(\frac 76)'=(\frac 76+\frac{p-7}6)/p=\frac 16.$ When $p\equiv 2\mod 3$, $[-\frac 76]_0=\frac{5p-7}6$, $(\frac 76)'=\frac 56.$

Fix a prime $p>2$. We now recall the reflection formula for the $p$-adic Gamma function. It says for $x\in \Z_p$, \begin{equation}\label{eq:reflect}\G_p(x)\G_p(1-x)=(-1)^{a_0(x)}\end{equation} where $a_0(x)$ is unique integer in $[1,2,\cdots,p]$ such that $x\equiv a_0(x)\mod p$.  If $x\in \Z_p^\times $, $a_0(x)=[x]_0$. See \cite{Cohen} by Cohen for more properties of the $p$-adic Gamma function.
\medskip

Let $M:=M(\alpha,\beta;\l)$ be the least common denominator of all $a_i$'s, $b_j$'s and $\l$.  Let $\alpha':=\{ a_1',\cdots, a_n'\}$,  namely the image of $\alpha$ under the dash operation.

\begin{theorem}[Dwork, \cite{Dwork}] \label{thm:Dwork}Let $\alpha,\beta$ be two multi-sets of rational numbers of size $n$, where $a_i\in(0,1)$ and $\beta=\{1,\cdots,1\}$. Then   for any prime $p\nmid M(\alpha,\beta;\l)$,  integer $m\ge 1$, integers $t\ge s$
\begin{equation}\label{eq:Dwork}
F(\alpha,\beta;\l)_{mp^{s}-1}F(\alpha',\beta;\l^p)_{mp^{t-1}-1}\equiv F(\alpha',\beta;\l^p)_{mp^{s-1}-1}F(\alpha,\beta;\l)_{mp^{t}-1}\mod p^{s}.
\end{equation}
\end{theorem}When  $p\nmid F(\alpha,\beta;\l)_{p-1}$ in which case we say that $p$ is  \emph{ordinary} for $\{\alpha,\beta;\l\}$, there is a $p$-adid unit root $\mu_{\alpha,\beta;\l,p}\in \Z_p^\times$  such that for any integers $s,m>0$
\begin{equation}\label{eq:unit}
F(\alpha,\beta,\l)_{mp^{s}-1}/F(\alpha',\beta,\l^p)_{mp^{s-1}-1}\equiv \mu_{\alpha,\beta;\l,p} \mod p^s.
\end{equation}
\medskip

\begin{defn}
A multi-set $\alpha=\{a_1,\cdots,a_n\}$ of rational numbers is said to be \emph{defined over} $\Q$ if $$\prod_{i=1}^n(X-e^{2\pi i a_i})\in \Z[X].$$ A hypergeometric datum $\{\alpha,\beta;\l\}$ is said to be \emph{defined over} $\Q$ if both $\alpha,\beta$ are defined over $\Q$ and $\l\in \Q$.
\end{defn}

Below we focus on hypergeometric data satisfying the following  condition  \medskip
\begin{quote}
($\diamond$): \quad $\alpha=\{a_1,\cdots, a_{n}\}$  and $\beta=\{b_1,\cdots,b_{n-1},1\}$ where $0\le a_i<1$ and $0<b_j\le 1$,  both $\alpha$ and $\beta$ are defined over $\Q$, and  $a_i- b_j\notin \Z$ for each $i,j$.
\end{quote}\medskip

\begin{remark}\label{rem:1}
If $\alpha=\{a_1\cdots,a_n\}$ is defined over $\Q$ and each $0< a_i<1$ for each $i$, then $1-a_i\in \alpha$. Moreover, for any fixed prime $p$ not dividing the least common denominator of $a_i$'s,  the set $\alpha$ is closed under the dash operation defined in Definition \ref{def:dash}.
\end{remark}

For any holomorphic modular form $f$, we use $a_p(f)$ to denote its $p$th Fourier coefficient. For most of the listed modular forms in this paper, they will be identified by  their labels in L-Functions and Modular Forms Database (LMFDB) \cite{LMFDB}. For instance, $f_{8.6.a.a}$ denotes a level-8 weight 6 Hecke eigenform, $\eta(z)^8\eta(4z)^4+8\eta(4z)^{12}$ in terms of the Dedekind eta function, its label at LMFDB is $8.6.a.a$. \medskip

In \cite{LTYZ}, the author jointly with Tu,  Yui, and Zudilin proved a conjecture of Rodriguez-Villegas made in \cite{RV}.  Special cases of this conjecture has been obtained earlier by Kilbourn  \cite{Kilbourn} (based on \cite{AO00} by Ahlgren and Ono), McCarthy \cite{McCarthy-RV}, and Fuselier and McCarthy \cite{FM}. See \cite{Zudilin18} by Zudilin for a related background discussion.

\begin{theorem}[Long, Tu, Yui, and Zudilin \cite{LTYZ}]\label{thm:LTYZ}
Let $\beta=\{1,1,1,1\}$. For each given multi-set $\alpha=\{r_1,1-r_1,r_2,1-r_2\}$ where either each $r_i\in \{\frac12,\frac13,\frac14,\frac16\}$ or $(r_1,r_2)=(\frac 15,\frac 25), (\frac 18,\frac 38) $, $(\frac 1{10},\frac 3{10})$, $(\frac 1{12},\frac 5{12})$,  there exists an explicit weight 4 Hecke cuspidal eigenform depending on $\alpha$, denoted by $f_\alpha$ such that for each prime $p>5$,
$$\pFq{4}{3}{r_1&1-r_1&r_2&1-r_2}{&1&1&1}{1}_{p-1}\equiv a_p(f_\alpha) \mod p^3.$$
\end{theorem} \footnote{In fact the proof of this theorem implies for any $m\ge 1$, $F(\alpha,\beta;1)_{mp-1}\equiv a_p(f_\alpha)F(\alpha,\beta;1)_{m-1} \mod p^3$.}

This statement resembles the congruence \eqref{eq:unit} when both $m=s=1$ and $p$ being ordinary.  In fact reducing the statement of Theorem \ref{thm:LTYZ}  modulo $p$, rather than modulo $p^3$, can be shown using the theory of commutative formal group law. Any congruence that is stronger than what can be predicted by the commutative formal group law is called a \emph{supercongruence}. Supercongruences often reveal additional background symmetries such as complex multiplication. For instance, the background setting of Theorem \ref{thm:LTYZ} involves  Calabi-Yau threefolds defined over $\Q$ which are rigid in the sense that their $h^{2,1}$ hodge numbers equal 0. The power $p^3$ is somewhat expected in terms of the Atkin and Swinnerton-Dyer (ASD) congruences satisfied by the coefficients of weight-$k$ modular forms for noncongruence subgroups \cite{ASD,LL2,Scholl}. Two-term ASD congruences for the coefficients of weight-$k$ modular forms can be written the same way as Equation \eqref{eq:Dwork} with $p^s$ being replaced by $p^{(k-1)s}$.

The proof of Theorem \ref{thm:LTYZ} uses the finite hypergeometric functions (see Definition \ref{defn:Hp} below) originated in Katz's work \cite{Katz} and modified by McCarthy \cite{McCarthy-well} and
a $p$-adic perturbation method, which was used in  \cite{Long} and  described explicitly in \cite{LR} by the author and  Ramakrishna. The method was also used in other papers, we list a few here: \cite{Swisher} by Swisher, \cite{Liu} by Liu, \cite{He} by He, \cite{BS} by Barman and Saikia, and \cite{MP} by Mao and Pan. Essentially one regroups the the corresponding character sum  into a major term, that is $F(\alpha,\beta;1)_{p-1}-a_p(f_\alpha)$, and graded error terms. There are different ways to deal with the error terms. In \cite{LTYZ}, they are shown to be zero modulo $p^2$ using identities constructed from a residue-sum technique, which was  also used  in \cite{OSZ, Zudilin}. This approach may settle some new cases here. See \cite{MOS} by McCarthy, Osburn, Straub for other front of supercongruence results. See  \cite{Doran} by Doran et al. for writing local zeta functions of certain K3 surfaces using finite hypergeometric functions.

\medskip

It is conjectured by Mortenson (stated in \cite{FOP} by Frechette, Ono, Papanikolas) that for all primes $p>2$
\begin{equation}\label{eq:Mortenson}
\pFq{6}{5}{\frac 12&\frac 12&\frac 12&\frac 12&\frac 12&\frac 12}{&1&1&1&1&1}{1}_{p-1}\equiv a_p(f_{8.6.a.a}) \mod p^5.
\end{equation}The mod $p^3$ claim has been obtained by Osburn, Straub, Zudilin \cite{OSZ}. In \cite{BD}, Beukers and Delaygue proved that for each positive integer $n$, $\alpha=\{\frac 12,\cdots,\frac12\}, \beta=\{1,\cdots,1\}$ of size $n$,  a modulo $p^2$ supercongruence of the form \eqref{eq:unit} holds for $F(\alpha,\beta; (-1)^{n(p-1)/2})_{p-1}$  when $p$ is ordinary. Their proof uses the properties of hypergeometric differential equations. Further they conjectured that the corresponding Dwork type congruence \eqref{eq:Dwork} holds modulo $p^{2s}$ instead of $p^s$. See Conjecture 1.5 of \cite{BD}.

\medskip

In \cite{RR},  Roberts and Rodriguez-Villegas made a more general hypergeometric supercongruence conjecture, stated below in our notation, corresponding to hypergeometric motive defined over $\Q$ (see  \cite{RV-2} by  Rodriguez-Villegas and \cite{Roberts} by Roberts) such that $\beta$ consists of all 1 and $\l =\pm 1$. Their conjecture summarizes the pattern of Theorem \ref{thm:LTYZ} and the supercongruence \eqref{eq:Mortenson}.

\begin{conj}[Roberts and Rodriguez-Villegas \cite{RR}]Let $\alpha=\{\a_1,\cdots,a_n\}$, $\beta=\{1,\cdots,1\}$ be multisets satisfying condition ($\diamond$), $\l=\pm 1$ (note that by Remark \ref{rem:1}, $\alpha'=\alpha$). Let $A$ be the unique submotive of the hypergeometric motive corresponding to $\{\alpha,\beta;\l\}$ with hodge number $h^{0,n-1}(A)=1$ and $r$ the smallest positive integer such that $h^{r,n-1-r}(A)=1$. For any $p\nmid M(\alpha,\beta;\l)$ and ordinary for $\{\alpha,\beta;\l\}$, there is a $p$-adic unit $\mu_{\alpha,\beta;\l,p}$ depending on the hypergeometric datum such that for any integer $s\ge 1$
$$F(\alpha,\beta;\l)_{p^s-1}/F(\alpha,\beta;\l)_{p^{s-1}-1}\equiv \mu_{\alpha,\beta;\l,p} \mod p^{rs}.$$
\end{conj} Comparing with Dwork's congruence \eqref{eq:unit}, this is a refinement and the degree of the supercongruence, is given in terms of the gap between two hodge numbers of  $A$. Our main focus here is to investigate the supercongruence phenomenon when the  assumption of $\beta=\{1,\cdots,1\}$ being relaxed. We shall see that  a slight adjustment to  the hypergeometric parameters for the truncated series might be needed in the statement of the generic congruence, which is Theorem \ref{lem:1}. It is based on the definition of finite hypergeometricc sum and two technical Lemmas dealing with the discrepancy among the orders of $p$ appearing in Gauss sums and the corresponding rising factorials $(a)_k$. Our numeric computation further confirms that the parameter adjustment is needed for some numeric supercongruences listed in later part of this paper.\medskip

Based on Roberts and Rodriguez-Villegas' philosophy, supercongruences occur to hypergeometric motives with gaps in hodge number sequences. For this reason, we consider a few hypergeometric motives  defined over $\Q$  that are decomposable.  To look for supercongruence candidates, the strategy we use is based on the Galois perspective of classical hypergeometric functions outlined in \cite{Katz} by Katz, \cite{BCM} by Beukers, Cohen and Mellit and in \cite{WIN3X} by the author jointly with Fuselier, Ramakrishna, Swisher and Tu.
\section{From finite hypergeometric functions to truncated hypergeometric series}

  Following \cite{BCM}, we use the following definition  which corresponds to the normalized hypergeometric function $_n\mathbb F_{n-1}$ in \S4.1 of \cite{WIN3X} modified from hypergeometric functions defined over finite fields in \cite{Greene} by Greene.

\begin{defn}\label{defn:Hp}
Let $\{\alpha,\beta;\l\}$ be a hypergeometric datum defined over $\Q$. For any finite field of size $q=p^e$ where $p\nmid M(\alpha,\beta;\l)$ such that $a_i(q-1), b_j(q-1)\in \Z$ for all $i,j$, the finite hypergeometric function corresponding to $\{\alpha,\beta;\lambda\}$ over $\F_p$ is defined to be
\begin{equation}\label{eq:Hp}
H_q(\alpha,\beta;\lambda):=\frac 1{1-q}\sum_{k=0}^{q-2} \prod_{i=1}^n\frac{g(\omega^{k+(q-1)a_i})g(\omega^{-k-(q-1)b_i})}{g(\omega^{(q-1)a_i})g(\omega^{-(q-1)b_i})}\omega^k((-1)^n\lambda),
\end{equation}where $\omega$ represents an order $q-1$ multiplicative character of $\F_q$, and in this article we use the Teichmuller character of $\F_q^\times\rightarrow \C_p^\times$; for any multiplicative character $\chi$ of $\F_q^\times$, $$g(\chi):=\sum_{x\in \F_q} \chi(x)\Psi(x)$$ stands for its Gauss sum with respect to a non trivial additive character $\Psi$ of $\F_q$
\end{defn}
 Using the multiplication formulas  for Gauss sums,  when $\alpha,\beta$ are defined over $\Q$, this character sum  \eqref{eq:Hp} can be formulated in another form to rid of the assumption on $a_i(q-1), b_j(q-1)\in \Z$. See Theorem 1.3 of \cite{BCM} for details. \medskip

Now we  recall the following function used in \cite{LTYZ}. Define $$\nu(a,x):=-\left \lfloor \frac{x-a}{p-1}\right \rfloor=\begin{cases}0& \text{ if $0<a\le x<p$}\\ 1 &\text{ if $0<x<a<p$}. \end{cases} $$

To relate the finite hypergeometric sum $H_p(\alpha,\beta;\l)$ defined in \eqref{eq:Hp} to the truncated hypergeometric functions, one uses the Gross-Koblitz formula \cite{GK} which says for integer $k$
\begin{equation*}
g(\omega^{-k})=-\pi_p^{[k]_0}\G_p\(\left \{\frac k{p-1}\right \}\),
\end{equation*}where $\omega$ is the Teichmuller character of $\F_p^\times$, $\G_p(\cdot)$ is the $p$-adic Gamma function, $\pi_p$ is a fixed root of $x^{p-1}+p=0$ in $\C_p$, the additive character for the Gauss sum is $\Psi(x):=\zeta_p^{\text{Tr}_{\F_p}^{\F_q}(x)}$ of $\F_q$ where $\zeta_p$ is a primitive $p$th root of unity which is congruent to $1+\pi_p$ modulo $\pi_p^2$.  Also for a real number $x$, $\{x\}$ stands for its fractional part.

\medskip

Following the analysis (as \cite[et al.]{AO00,Kilbourn,McCarthy-RV,LTYZ}), we use the Gross-Koblitz formula to rewrite \eqref{eq:Hp} into the following theorem. When  $\beta=\{1,\cdots,1\}$ the details of a proof is already given in Section 4 of \cite{LTYZ}. Another reference is Beukers' paper \cite{Beukers}.
\begin{theorem}\label{lem:0}Assume that $\alpha,\beta$ are two multi-sets satisfying condition $(\diamond)$. Then for any prime $p\nmid M(\alpha,\beta;\l)$
\[
H_p(\alpha,\beta;\l)=\frac{1}{1-p}\sum_{k=0}^{p-2}\prod_{i=1}^n\frac{\G_p\left (\left \{ a_i-\frac{k}{p-1}  \right \} \right )}{\G_p(a_i)}\frac{\G_p(1-\{b_i\})}{\G_p\left (  1-\left \{b_i+\frac k{p-1}\right  \} \right)} (-p)^{e_{\alpha,\beta}(k)}\omega(\l)^k,
\] where 
$$e(k)=e_{\alpha,\beta}(k):=\sum_{i=1}^n -\left \lfloor a_i-\frac{k}{p-1}\right \rfloor-\left \lfloor \frac{k}{p-1}+\{b_i\}  \right \rfloor,$$ is a step function of $k$ when $k$ is viewed as a variable varying  from $0$ to $p-1$. The jumps  only happen at values $(p-1)a_i$ or $(p-1)(1-\{b_j\})$.
\end{theorem}

Comparing with Beukers' paper \cite{Beukers}, the function $e(k)=\Lambda(p-1-k)$ where $\Lambda(k)$ is given in Definition 1.4 of \cite{Beukers}. The discrepancy is caused by the choice of the generator for the group of multiplicative characters of $\F_p^\times$. In addition we apply the reflection formula for $\G_p(\cdot)$   to get a rising factorial of $b_i$ in the denominator, as stated  in Theorem \ref{lem:1} below.

We further define some terms here which are useful for our discussion.

\begin{defn}For any given $\alpha,\beta$ defined over $\Q$, define
\begin{equation*}
s(\alpha,\beta):=\min\{e(k)\mid 0\le k<p-1\};\end{equation*} and the bottom interval as
\begin{equation*}
I(\alpha,\beta):=\{0\le x<p-1 \mid e(x)=s(\alpha,\beta)\}
\end{equation*} and the \emph{weight} function as the difference between the largest and smallest $e(k)$ values
\begin{equation*}
w(\alpha,\beta):=\max\{e(k)\mid 0\le k\le p-2\}-\min\{e(k)\mid 0\le k< p-1\}.\end{equation*}
\end{defn}

Below we only consider $\l=\pm 1$.
By  the way we define $s(\alpha,\beta)$ and Theorem \ref{lem:0},  $p^{\max\{0,-s(\alpha,\beta)\}}H_p(\alpha,\beta;\l)$ is $p$-adically integral. Also the value of $p^{\max\{0,-s(\alpha,\beta)\}}H_p(\alpha,\beta;\l)$ modulo $p$ only relies on the contribution from the $k$th terms where $k$ is in the bottom interval $I(\alpha,\beta)$. Namely
\begin{multline}\label{eq:I}p^{\max\{0,-s(\alpha,\beta)\}}H_p(\alpha,\beta;\pm 1) \\ \equiv p^{\max\{0,-s(\alpha,\beta)\}}\sum_{k, e(k)=s(\alpha,\beta)} \prod_{i=1}^n\frac{\G_p\left (\left \{ a_i-\frac{k}{p-1}  \right \} \right )}{\G_p(a_i)}\frac{\G_p(1-\{b_i\})}{\G_p\left (1- \left \{ b_i+\frac k{p-1}\right  \} \right)}(-p)^{e_{\alpha,\beta}(k)}(\pm 1)^{k}\\ \mod p.\end{multline}
\medskip

The next lemmas are useful for the proof of our main Theorem \ref{lem:1}.
\begin{lemma}\label{lem:nu-value}Let $p$ be a fixed prime. Then
for any $a\in \Q\cap (0,1)\cap \Z_p^\times$, non negative integer $ k<p$ $$\nu(k,(p-1)a')=\nu(k,[-a-\nu(k,(p-1)a)]_0),$$ where $a'$ is the image of $a$ under the dash operation as before.
\end{lemma}
\begin{proof}Recall that $a+[-a]_0=pa'$ and both $a,a'$ are in $(0,1)$ by our assumption.
If $k\le (p-1)a'$, i.e. $k- (pa'-a+a-a')=k-[-a]_0-(a-a')\le 0$, then $k\le [-a]_0$ as the difference is an integer less than $|a-a'|<1$. Thus the values on both sides are 0.

If $k>(p-1)a'$, then $k-[-a-\nu(k,(p-1)a)]_0=k-[-a-1]_0=k+1-[-a]_0=k+1-pa'+a>0$. Thus both  values are 1.
\end{proof}

To use Theorem \ref{lem:0}   to link $H_p(\alpha,\beta;\pm 1)$ to a truncated hypergeometric series, the
 next Lemma will be helpful. It is mainly about converting the $p$-adic unit part of $g(\omega^k+(p-1)a)/g(\omega^{(p-1)a})$ to a rising factorial.
\begin{lemma}\label{lem:nu}Let $p>2$ be a fixed prime.
For $a\in \Q\cap [0,1)\cap \Z_p^\times$, non negative integer $ k<p-1$, modul $p$
\begin{eqnarray*} \G_p \left (\left \{a- \frac{k}{p-1} \right\}\right )/\G_p(a)&\equiv& \begin{cases}
C\cdot (a)_k  \quad \text{if}  \quad k\le (p-1)a \\C\cdot (a+1)_k \quad \text{if}\quad (p-1)a<k<p-1
\end{cases} \\
&= & C (a+\nu(k,(p-1)a))_k ,\end{eqnarray*}
where $\displaystyle C= \frac{(-1)^k}{(-p)^{\nu(k,(p-1)a)}}\frac{(ap)^{\nu(k,(p-1)a)}}{(a'p)^{\nu(k,(p-1)a')}}$.

\end{lemma}
\begin{proof}We will prove the equivalent form \begin{multline*}(-p)^{\nu(k,(p-1)a)} \G_p \left (\left \{a- \frac{k}{p-1} \right\}\right )/\G_p(a) \\ \equiv(-1)^k\frac{(ap)^{\nu(k,(p-1)a)}}{(a'p)^{\nu(k,(p-1)a')}}(a+\nu(k,(p-1)a))_k \mod p^{1+\nu(k,(p-1)a)}.
\end{multline*}

If $k\le (p-1)a$, $\nu(k,(p-1)a)=0$. Thus $$(-p)^{\nu(k,(p-1)a)} \G_p \left (\left \{a- \frac{k}{p-1} \right\}\right )/\G_p(a)=\G_p\(a-\frac k{p-1}\)/\G_p(a)\equiv \G_p(a+k)/\G_p(a) \mod p.$$

If $k> (p-1)a$, $\nu(k,(p-1)a)=1$. Thus
\begin{multline*}(-p)^{\nu(k,(p-1)a)} \G_p \left (\left \{a- \frac{k}{p-1} \right\}\right )/ \G_p(a)\\=(-p)\G_p\(1+a-\frac k{p-1}\)/\G_p(a)\equiv (-p)\G_p(1+a+k)/\G_p(a) \mod p^2.
\end{multline*}As $\G_p(a+1)=(-a)\G_p(a)$ for $a\in \Z_p^\times$. $(-p)\G_p(1+a+k)/\G_p(a)=(ap)\G_p(1+a+k)/\G_p(1+a)$.

Note that for any $0\le k<p$, and $c\in \Z_p^\times$ if $k\le [-c]_0$, $\G_p(c+k)/\G_p(1+c)=(-1)^k(c)_k$ and when $k> [-a]_0$, $\G_p(c+k)/\G_p(c)=(-1)^k(c'p)^{-1}(c)_k$ as among $c,c+1,\cdots,c+k-1$ one of them is a multiple of $p$, which equals $c+[-c]_0=c'p$ by the definition of the dash operation. Putting together we have

\begin{multline*}(-p)^{\nu(k,(p-1)a)} \G_p \left (\left \{a- \frac{k}{p-1} \right\}\right )/\G_p(a) \\ \equiv (ap)^{\nu(k,(p-1)a)}  \G_p(a+k+{\nu(k,(p-1)a)} )/\G_p(a+{\nu(k,(p-1)a)} ) \mod p^{1+\nu(k,(p-1)a)} \\=(-1)^k\frac{(ap)^{\nu(k,(p-1)a)}}{(a'p)^{\nu(k,[-a-\nu(k,(p-1)a)]_0)}}(a+\nu(k,(p-1)a))_k \mod p^{1+\nu(k,(p-1)a)}.
\\=(-1)^k\frac{(ap)^{\nu(k,(p-1)a)}}{(a'p)^{\nu(k,(p-1)a')}}(a+\nu(k,(p-1)a))_k \mod p^{1+\nu(k,(p-1)a)}.
\end{multline*}The last claim follows from Lemma \ref{lem:nu-value}.
\end{proof}

\begin{theorem}\label{lem:1}Assume that
 $\alpha,\beta$ are two multi-sets satisfying the $\diamond$ condition  such that

 1) $s(\alpha,\beta)<0$;

 2) $I(\alpha,\beta) $ is connected.

 Let $\l\in \Q^\times$ and p be any odd prime  not dividing $M(\alpha,\beta;\l)$, then

\begin{equation}\label{eq:generic}p^{-s(\alpha,\beta)}H_p(\alpha,\beta;\l)\equiv p^{-s(\alpha,\beta)} \cdot F(\hat
\alpha, \breve {\beta}; \l)_{p-1} \mod p,\end{equation}where  $\hat \alpha=\{\hat a_1,\cdots, \hat a_n\}$, $\breve{\beta}=\{\breve b_1,\cdots, \breve b_n\}$ with $\hat a_i=a_i+\nu(k,(p-1)a_i)$ and $\breve b_i=b_i+\left \lfloor \frac{k}{p-1}+\{1-b_i\}  \right \rfloor$ for any integer $k\in I(\alpha,\beta)$.

\end{theorem}
\begin{proof}Note that $\omega(\l)=\l \mod p$.
Under the assumption that $I(\alpha,\beta)$ is connected, the values of $\hat a_i$ and $\breve b_j$ are independent of the choice of $k$ within the interval $I(\alpha,\beta)$. By \eqref{eq:I}, in the formula $H_p(\alpha,\beta;\pm 1)$ in Theorem \ref{lem:0}, we only need to consider $k\in I(\alpha,\beta)$. In this formula we separate
$$ \prod_{i=1}^n\frac{\G_p\left (\left \{ a_i-\frac{k}{p-1}  \right \} \right )}{\G_p(a_i)}\frac{\G_p(1-\{b_i\})}{\G_p\left (   1-\left \{b_i+\frac k{p-1}\right  \} \right)} (-p)^{e_{\alpha,\beta}}$$ into a product of two parts, one involves
$$\frac{\G_p\left (\left \{ a_i-\frac{k}{p-1}  \right \} \right )}{\G_p(a_i)}(-p)^{ -\left \lfloor a_i-\frac{k}{p-1}\right \rfloor}=\frac{\G_p\left (\left \{ a_i-\frac{k}{p-1}  \right \} \right )}{\G_p(a_i)}(-p)^{\nu(k,(p-1)a_i)},$$ to which one can apply use Lemma \ref{lem:nu} directly  to get
$$\prod_{i=1}^n\frac{\G_p\left (\left \{ a_i-\frac{k}{p-1}  \right \} \right )}{\G_p(a_i)}(-p)^{ -\left \lfloor a_i-\frac{k}{p-1}\right \rfloor}=(-1)^{nk}\prod_{i=1}^n(\hat a_i)_k \mod p^{1+\sum_{i=1}^n\nu(k,(p-1)a_i)},$$ as the set $\alpha=\{a_1,\cdots,a_n\}$ is closed under the dash operation under our assumption.

The other part involves
$$\frac{\G_p(1-\{b_i\})}{\G_p\left (   1-\left \{b_i+\frac k{p-1}\right  \} \right)}(-p)^{-\left \lfloor \frac{k}{p-1}+\{b_i\}  \right \rfloor},$$ which will be divided into two cases.

Firstly,  we assume there is an integer  $k=(p-1)(1-b_i)$ within $I(\alpha,\beta)$ for some $b_i\notin \Z$ and hence $1-b_i\in \beta$. Thus $(1-b_i)'=1-b_i$ in this case. Note that the multiset  obtain from removing all $1$ and $1-b_i$ from $\beta$ is still close under the dash operation. In this case $\left \lfloor \frac{k}{p-1}+b_i  \right \rfloor=1$ and the reciprocal of the above is

\begin{multline*}(-p)^{\left \lfloor \frac{k}{p-1}+\{b_i\}  \right \rfloor}\G_p\( 1-\left \{b_i+ \frac{k}{p-1} \right\}\)/\G_p(1-\{b_i\})=(-p)\G_p(1)/\G_p(1-b_i)=p\G_p(0
)/\G_p(1-b_i)\\ \equiv p \G_p(1-b_i+k)/\G_p(1-b_i)= (1-b_i)p\G_p(2-b_i+k)/\G_p(2-b_i) \mod p^2,\end{multline*}where $(1-b_i)p\G_p(2-b_i+k)/\G_p(2-b_i)=(-1)^k(2-b_i)_k$

For other $k$ in $I(\alpha,\beta)$ which is not of the form $(p-1)(1-b_i)$ and $b_i\notin \Z$, $\left \lfloor \frac{k}{p-1}+\{b_i\}  \right \rfloor=\nu(k,(p-1)(1-\{b_i\}))$ and $ 1-\left \{b_i+\frac k{p-1}\right  \} = \left \{1-\{b_i\}-\frac k{p-1}\right  \}$ when $0\le k<p-1$ and $k\neq (p-1)(1-\{b_i\})$. Thus one can also apply Lemma \ref{lem:nu} as $$\frac{\G_p(1-\{b_i\})}{\G_p\left (   1-\left \{b_i+\frac k{p-1}\right  \} \right)}(-p)^{-\left \lfloor \frac{k}{p-1}+\{b_i\}  \right \rfloor}=(-p)^{-\nu(k,(p-1)(1-\{b_i\}))}\frac{\G_p(1-\{b_i\})}{\G_p\left (  \left \{ 1- \{b_i\} -\frac k{p-1}\right  \} \right)} .$$ Similar to the case dealing with $a_i$, they contribute $(-1)^k(\breve b_i)_k$ in the denominators.
\end{proof}

Theorem \ref{lem:1} is a generic congruence, could be compared with Theorem \ref{thm:Dwork} in the ordinary case when $m=s=1$. We  now give two examples, and state a corresponding numeric supercongruence  modulo $p^5$.

\begin{example}\label{eg:1}
Let $\alpha=\{\frac12,\frac12,\frac12,\frac12,\frac13,\frac 23\}$, $\beta=\{1,1,1,1,\frac16,\frac 56\}$,  then $s(\alpha,\beta)=-1$, $w(\alpha,\beta)=6$, and the bottom interval $I(\alpha,\beta)=\left [\frac{p-1}6,\frac{p-1}3\right ]$. Thus $\hat \alpha=\alpha$ and $\breve  \beta=\{1,1,1,1,\frac76,\frac 56\}$.  In this case Theorem \ref{lem:1} says each prime $p\nmid M(\alpha,\beta;\l)$ \begin{equation}\label{eq:6}p\cdot \,\pFq{6}{5}{\frac 12&\frac 12&\frac 12&\frac 12&\frac 13&\frac 23}{&1&1&1&\frac76&\frac56}{\l }_{p-1}\equiv p\cdot H_p(\alpha,\beta;\l) \mod p.
\end{equation} We plot the value of $e(pk)$ in the following graph with $k$ ranges from 0 to 1.
\medskip

\begin{center}
\includegraphics[scale=0.3,origin=c]{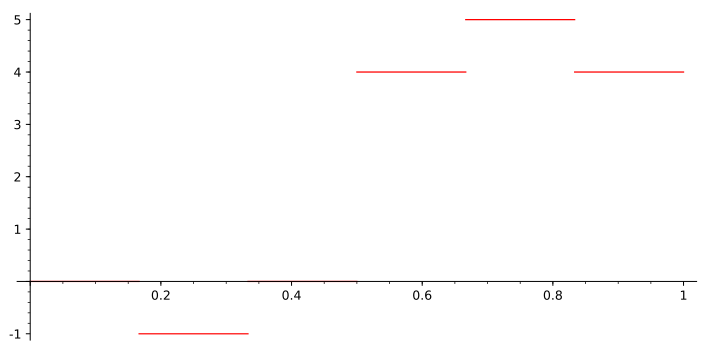}

{$e(pk)$ values for $\alpha=\{\frac12,\frac12,\frac12,\frac12,\frac13,\frac 23\}$, $\beta=\{1,1,1,1,\frac16,\frac 56\}$ }

\end{center}
\medskip

\end{example}

\begin{example}
 If  $\alpha=\{\frac12,\frac12,\frac16,\frac 56\}$, $\beta=\{1,1,\frac13,\frac 23\}$, then $s(\alpha,\beta)=0$, $w(\alpha,\beta)=2$,  and bottom interval being $I(\alpha,\beta)=\left[0,\frac{p-1}6\right ]\cup \left [\frac{p-1}3,\frac{p-1}2\right].$

 \medskip

 \begin{center}
\includegraphics[scale=0.3,origin=c]{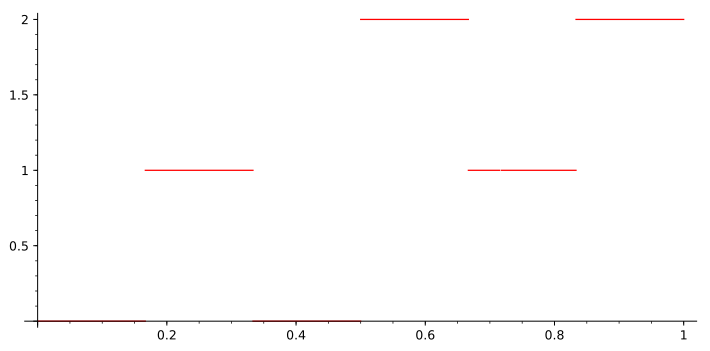}

{$e(pk)$ values for $\alpha=\{\frac12,\frac12,\frac16,\frac 56\}$, $\beta=\{1,1,\frac13,\frac 23\}$ }

\end{center}
\end{example}

\medskip

Now we turn our attention to supercongruences. Note \eqref{eq:6} is a generic modulo $p$ congruence which holds for $\l \in \Q^\times$. When $\l=\pm 1$, the congruence might be stronger.  One of them, listed below,  corresponds to Example \ref{eg:1} with $\l=1$.
\begin{conj}\label{conj:1}For each prime $p> 5$,

1). on the perspective of character sums
\begin{equation*}\label{eq:Hp-Conj1}p\cdot H_p\left (\left\{\frac12,\frac12,\frac12,\frac12,\frac13,\frac23\right\},\left \{1,1,1,1,\frac76,\frac56\right\};1\right)=\left ( \frac {2}p\right )a_p(f_{64.6.a.f})+\left ( \frac {-3}p\right )p\cdot a_p(f_{36.4.a.a})+\left (\frac{3}p\right )p^2;
\end{equation*}

2). on the perspective of supercongruences
\begin{equation*}\label{eq:1}
p\cdot \, \pFq{6}{5}{\frac 12&\frac 12&\frac 12&\frac 12&\frac 13&\frac 23}{&1&1&1&\frac76&\frac56}{1}_{p-1}\equiv \left ( \frac {2}p\right )a_p(f_{64.6.a.f}) \mod p^5,
\end{equation*}where $\left(\frac{\cdot}p\right)$ is the Legendre symbol.
\end{conj}The second claim is a refinement of \eqref{eq:6}  and is comparable with Mortenson's conjecture \eqref{eq:Mortenson}.

We will explain how it is found and  list a few other cases in the last section. Note that Z.-W. Sun has many open conjectures on congruences. Interested readers are referred to \cite{Sun}.

\section{From classical hypergeometric formula to Conjecture \ref{conj:1}}
 A key motivation for both \cite{WIN3X} and the present article  is to turn classical hypergeometric formulas into useful geometric guidance.  For instance,   here is a  formula of Whipple, see Theorem 3.4.4 of the book \cite{AAR} by Andrews, Askey and Roy when both sides terminate.
\begin{multline}\label{eq:1/2-6}
\pFq{7}{6}{a&1+\frac a2&c&d&e&f&g}{&\frac a2& 1+a-c&1+a-d&1+a-e&1+a-f&1+a-g}{1}\\= \frac{\G(1+a-e)\G(1+a-f)\G(1+a-g)\G(1+a-e-f-g)}{\G(1+a)\G(1+a-f-g)\G(1+a-e-f)\G(1+a-e-g)}\\ \times\pFq{4}{3}{a&e&f&g}{&e+f+g-a&1+a-c&1+a-d}{1}
\end{multline}

In the finite field analogue, the $_7F_6$  is reducible as $1+\frac a2$ and $\frac a2$ correspond to the same multiplicative character. Thus the $_7F_6$ is linked to
\begin{equation}\label{eq:6F5}
\pFq{6}{5}{a&c&d&e&f&g}{& 1+a-c&1+a-d&1+a-e&1+a-f&1+a-g}{1}.
\end{equation}

We next pick $a,c,d,e,f,g$ such that the hypergeometric data for both \eqref{eq:6F5} and the $_4F_3$   are both defined over $\Q$.
For example, we can let $a=c=d=e=f=g=\frac 12$,  so the hypergeometric datum for \eqref{eq:6F5} is $$ \left \{ \left \{\frac 12,\frac 12,\frac12,\frac12,\frac 12,\frac 12\right \},\{1,1,1,1,1,1\} ;1 \right \}.$$  In this case $H_p(\alpha,\beta;1)$ has an  explicit decomposition as follows, which  was conjectured by Koike and proved by Frechette, Ono and Papanikolas in \cite{FOP}. For any odd prime $p$ $$H_p(\alpha,\beta;1)=a_p(f_{8.6.a.a})+a_p(f_{8.4.a.a})\cdot p+\left (\frac{-1}p\right)p^2.$$ The factor $a_p(f_{8.4.a.a})\cdot p$ is associated with the right hand side of \eqref{eq:1/2-6} where $p$ corresponds to the Gamma values and the modular form $f_{8.4.a.a}$  arises from the  $_4F_3$ as for odd primes $p$
$$H_p\left(\left \{ \frac12,\frac12,\frac 12,\frac 12 \right\},\{1,1,1,1\} ;1\right)=a_p(f_{8.4.a.a})+p.$$ In other words \eqref{eq:1/2-6} implies a decomposition of the character sum corresponding to \eqref{eq:6F5}.

\medskip

Next we let $a=c=d=e=\frac12$ and $f=\frac 13,g=\frac23$. The datum for \eqref{eq:6F5} is $$H_1=\left \{\left\{\frac12,\frac12,\frac12,\frac12,\frac13,\frac23\right\},\left \{1,1,1,1,\frac76,\frac56\right\};1\right\},$$ the same datum for Conjecture \ref{conj:1}. The datum for the $_4F_3$ is $H_2=\left \{\left\{\frac12,\frac12,\frac13,\frac23\right\},\left \{1,1,1,1\right\};1\right\}.$ By Theorem 2 of \cite{LTYZ},
$$H_p\left(\left\{\frac12,\frac12,\frac13,\frac23\right\},\left \{1,1,1,1\right\};1\right)=a_p(f_{36.4.a.a})+\left (\frac 3p\right)p.$$

Recall  the local zeta function for a hypergeometric data $\{\alpha, \beta;\lambda \}$ defined over $\Z_p$ as follows.  $$P_p(\{\alpha,\beta;\l \},T):=\exp\left(\sum_{s=1}^\infty \frac{H_{p^s}(\alpha,\beta;\l)}sT^s\right)^{(-1)^n},$$ which is known to be a polynomial with constant 1. Assume $P_p(\{\alpha,\beta;\l \},T)=\prod_i(\mu_i T-1)$. When $\l=1$ and $n$ even, the structure of the hypergeometric motive is usually degenerated in the sense one of the $\mu_i$'s' has smaller absolute value as a complex number.  For instance, corresponding to Conjecture \ref{conj:1}, when $p\nmid M$, $P_p(\{\alpha,\beta;1 \},T)$ is a polynomial of degree 5. Among the absolute values of $\mu_i$, four of them are $p^{5/2}$ and one of them is $p^2$, corresponding to $\left(\frac{-3}p\right)p^2$ in the formula for $H_p$.

In \texttt{Magma}, there is an implemented hypergeometric package by Watkins (see \cite{Watkins} for its documentation). In particular one can use it to compute $P_p(\{\alpha,\beta;\l \},T)$. When $\l=1$, the implemented command has the singular linear factor  removed. For instance, typing the following into \texttt{Magma} \medskip

\noindent \texttt{H1 := HypergeometricData([1/2,1/2,1/2,1/2,1/3,2/3],[1,1,1,1,1/6,5/6]);}

\noindent \texttt{EulerFactor(H1,1,5);  ([Here we let $p=5$.])}

\medskip

\noindent The output is \medskip

   \noindent \texttt{ <$9765625*\$.1^4 + 112500*\$.1^3 + 1390*\$.1^2 + 36*\$.1 + 1$>} \medskip

   Moreover, the negation of the linear coefficient above, $-36$, equals $\left ( \frac {2}5\right )a_5(f_{64.6.a.f})+\left ( \frac {-3}5\right )5\cdot a_5(f_{36.4.a.a})$.

\medskip

Next we use $p=7$ to illustrate  the first claim of Conjecture \ref{conj:1}. Typing \medskip

\noindent \texttt{Factorization(EulerFactor(H1,1,7));}\medskip

We get the following decomposition\medskip

 \noindent   <$16807*\$.1^2 - 56*\$.1 + 1,1$>,

 \noindent   <$16807*\$.1^2 + 88*\$.1 + 1, 1$>

\medskip
To see the role of the second hypergeometric datum $H_2$, use \medskip

\noindent \texttt{H2 := HypergeometricData([1/2,1/2,1/3,2/3],[1,1,1,1]);}

\noindent \texttt{Factorization(EulerFactor(H2,1,7));} \medskip

We get

\medskip

  \noindent  <$343*\$.1^2 - 8*\$.1 + 1, 1$>.

\medskip

Notice that if we let $f(x)=343x^2 - 8x + 1$ as above, then $f(px)=f(7x)=16807x^2 - 56x + 1$ coincides with the first  quadratic factor of \texttt{Factorization(EulerFactor(H1,1,7))}.

\medskip

The following command in \texttt{Magma} allows us to produce a sequence. \medskip

\noindent \texttt{[-Coefficient(EulerFactor(H1,1,p),1)+LegendreSymbol(-3, p)*p*Coefficient(EulerFactor(H2,1,p),1) : p in PrimesUpTo(31) | p ge 7];}  \medskip

which gives
\medskip

\noindent \texttt{[$-88, 540, -418, 594, 836, -4104, -594, 4256$]} \medskip

It coincides with $\left ( \frac {2}p\right )a_p(f_{64.6.a.f})$ when $p$ varies from 7 to 31, confirming the roles of two modular forms in  the first statement of Conjecture \ref{conj:1}, we then compute the sign for the $p^2$ term and check the second statement numerically. \medskip

It is natural to ask whether the parameter modification, i.e. from $\frac 16$ to $\frac 76$ is necessary in Conjecture \ref{conj:1}.   Numerically we have for each prime $p\ge7$,

$$-2p\cdot \, \pFq{6}{5}{\frac 12&\frac 12&\frac 12&\frac 12&\frac 13&\frac 23}{&1&1&1&\frac16&\frac56}{1}_{p-1}\equiv \left ( \frac {2}p\right )a_p(f_{64.6.a.f}) \mod p.$$ The power $p^1$ is sharp for a generic $p$ and the left hand side has an additional multiple of $-2$ which can be explained as follow. Letting $a=c=d=\frac 12$, $f=\frac 13,g=\frac 23$ and $e=\frac{1-p}2$ to Equation \eqref{eq:1/2-6} we have
\begin{equation*}
\pFq{7}{6}{\frac 12&\frac 54&\frac 12&\frac 12&\frac 12&\frac 13&\frac 23}{&\frac 14&1&1&1&\frac76&\frac56}{1}_{p-1}\in \Z_p
\end{equation*}which implies
$$-2p\cdot \, \pFq{6}{5}{\frac 12&\frac 12&\frac 12&\frac 12&\frac 13&\frac 23}{&1&1&1&\frac16&\frac56}{1}_{p-1}\equiv p\cdot \, \pFq{6}{5}{\frac 12&\frac 12&\frac 12&\frac 12&\frac 13&\frac 23}{&1&1&1&\frac76&\frac56}{1}_{p-1} \mod p.$$

\bigskip

\begin{remark}
For the hypergeometric datum $\{\{\frac 12,\frac12,\frac12,\frac12,\frac13,\frac16\},\{1,1,1,1,\frac14,\frac34\},1\}$ defined over $\Q$ which cannot be obtained from specializing parameters in Equation \eqref{eq:1/2-6}. Numerically its Euler $p$ factors are irreducible over $\Z$ for many $p$'s.
\end{remark}
\bigskip

\section{More Numeric findings}
\subsection{A few other $_6F_5(1)$ cases}


\subsubsection{$\alpha=\{\frac 12,\frac 12,\frac13,\frac23,\frac13,\frac 23\}, \beta=\{1,1,\frac16,\frac56,\frac16,\frac 56\} $}

\begin{center}
\includegraphics[scale=0.25,origin=c]{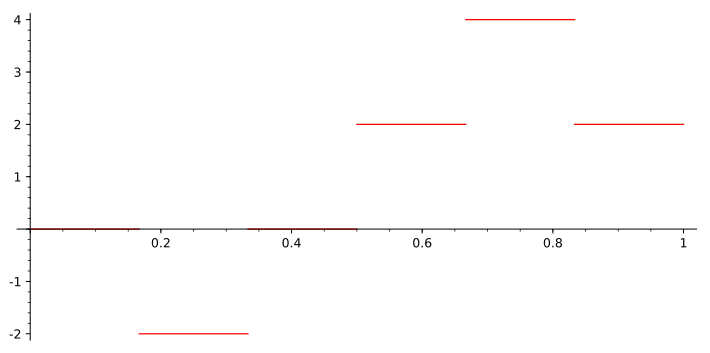}

{$e(pk)$ for $\alpha=\{\frac 12,\frac 12,\frac13,\frac23,\frac13,\frac 23\}$, $\beta=\{1,1,\frac16,\frac 56,\frac16,\frac 56\}$ }

\end{center}
\medskip

 Here $w(\alpha,\beta)=6$, $s(\alpha,\beta)=-1$ and $I(\alpha,\beta)$ is connected, $\hat \alpha=\alpha$ and $\breve  \beta=\{1,1,\frac76,\frac56,\frac76,\frac 56\}$.
Similar to the previous discussion, we specify $a=c=\frac12$, $d=f=\frac13,e=g=\frac 23$ in Equation \eqref{eq:1/2-6}.
\begin{conj}
There is a weight 6 modular form  $f_{48.6.a.c}$ such that for all primes $p\ge 7$,
$$p^2\cdot \, \pFq{6}{5}{\frac 12&\frac 12&\frac 13&\frac 23&\frac 13&\frac 23}{&1&\frac76&\frac56&\frac76&\frac56}{1}_{p-1}\equiv \left (\frac{-1}p\right)a_p(f_{48.6.a.c}) \mod p^4.$$
\end{conj}

\medskip

\subsubsection{ $\alpha=\{\frac12,\frac12,\frac12,\frac12,\frac16,\frac 56\}$, $\beta=\{1,1,1,1,\frac13,\frac 23\}$}

\begin{center}
\includegraphics[scale=0.25,origin=c]{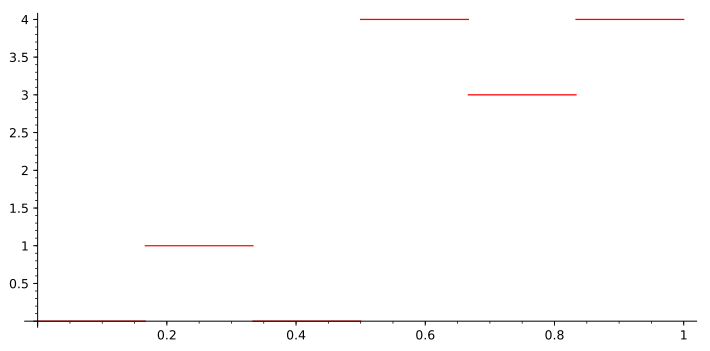}

{$e(pk)$  for $\alpha=\{\frac12,\frac12,\frac12,\frac12,\frac16,\frac 56\}$, $\beta=\{1,1,1,1,\frac13,\frac 23\}$ }

\end{center} In this case $w(\alpha,\beta)=4$, $s(\alpha,\beta)=0$, 
and $I(\alpha,\beta)$ consists of two disjoint intervals. Anyway, we proceed similarly as the previous cases. In Equation \eqref{eq:1/2-6}, we let $a=c=d=e=\frac 12$, $f=\frac 16,g=\frac 56$.  Numerically we have
\begin{conj}For primes $p\ge 5$, $\alpha=\left \{\frac 12,\frac 12,\frac 12,\frac 12,\frac 16,\frac 56\right \}$, $\breve  \beta=\left \{1,1,1,1,\frac 43,\frac 23\right \}$
\[
H_p(\alpha,\breve  \beta;1)=\left ( \frac 2p\right)a_p(f_{64.4.a.d})+\left (\frac{-3}p\right)a_p(f_{72.4.a.b})+\left (\frac{3}p\right)p.
\]
\begin{equation*}\pFq{6}{5}{\frac 12&\frac 12&\frac 12&\frac 12&\frac 16&\frac 56}{&1&1&1&\frac43&\frac23}{1}_{p-1}\equiv \left ( \frac 2p\right)a_p(f_{64.4.a.d}) \mod p^3.\end{equation*}
\end{conj}

\medskip



\subsection{Some $_4F_3(1)$ cases}

In addition, we found numerically a few $_4F_3(1)$ supercongruences for hypergeometric motives. We  first list 6 cases with weights $w(\alpha,\beta)=4$. Numerically they satisfy supercongruences analogous to the statement of Theorem \ref{thm:LTYZ}.

\begin{conj}For each prime $p\ge 7$,

\begin{equation}p\cdot \pFq{4}{3}{\frac12&\frac12&\frac13&\frac23}{&1&
\frac54&\frac34}{1}_{p-1}\equiv \left (\frac {3}p\right )a_p(f_{48.4.a.c}) \mod p^3.\end{equation}

\begin{equation}p\cdot \pFq{4}{3}{\frac12&\frac12&\frac13&\frac23}{&1&\frac76&\frac56}{1}_{p-1}\equiv \left (\frac {-1}p\right )a_p(f_{48.4.a.c}) \mod p^3.\end{equation}

\begin{equation}p\cdot \pFq{4}{3}{\frac12&\frac12&\frac14&\frac34}{&1&\frac76&\frac56}{1}_{p-1}\equiv  a_p(f_{48.4.a.c}) \mod p^3.\end{equation}

\begin{equation}p\cdot \pFq{4}{3}{\frac12&\frac12&\frac12&\frac12}{&1&\frac43&\frac23}{1}_{p-1}\equiv a_p(f_{24.4.a.a}) \mod p^3.\end{equation}

\begin{equation}p\cdot \pFq{4}{3}{\frac12&\frac12&\frac12&\frac12}{&1&\frac76&\frac56}{1}_{p-1}\equiv a_p(f_{12.4.a.a}) \mod p^3.\end{equation}

\begin{equation}p\cdot \pFq{4}{3}{\frac12&\frac12&\frac12&\frac12}{&1&\frac54&\frac34}{1}_{p-1}\equiv a_p(f_{64.4.a.b}) \mod p^3.\end{equation}
\end{conj}

There are two cases with $w(\alpha,\beta)=2$ relating to weight-2 modular forms. They can be obtained using the approach of \cite{LTYZ}.
 For each prime $p> 5$, 
\begin{equation*} \pFq{4}{3}{\frac12&\frac12&\frac16&\frac56}{&1&\frac43&\frac23}{1}_{p-1}\equiv  a_p(f_{24.2.a.a}) \mod p.\end{equation*}

\begin{equation*} \pFq{4}{3}{\frac12&\frac12&\frac14&\frac34}{&1&\frac43&\frac23}{1}_{p-1}\equiv  \left (\frac{-1}p\right)a_p(f_{24.2.a.a}) \mod p.\end{equation*}

\subsection{Some $_5F_4(-1)$ cases}

Here is  another evaluation formula of Whipple, see Theorem 3.4.6 of \cite{AAR}.
\begin{multline}\label{eq:Whipple2}
\pFq{6}{5}{a&a/2+1&b&c&d&e}{&a/2&a-b+1&a-c+1&a-d+1&a-e+1}{-1}\\
=\frac{\G(a-d+1)\G(a-e+1)}{\G(a+1)\G(a-d-e+1)}\pFq{3}{2}{a-b-c+1&d&e}{&a-b+1&a-c+1}{1}
\end{multline}

\subsubsection{$\alpha=\{\frac12,\frac12,\frac12,\frac12,\frac12\},\beta=\{1,1,1,1,1\}$} For this case, $s(\alpha,\beta)=0$, $w(\alpha,\beta)=5$,  $I(\alpha,\beta)=\left [0,\frac{p-1}2\right ]$, $\hat \alpha=\alpha,\breve  \beta=\beta$.

Letting $a=b=c=d=e=\frac12$ in  Equation \eqref{eq:Whipple2}, the left and right hand sides correspond to the hypergeometric data $$H_5=\left \{\left\{\frac12,\frac12,\frac12,\frac12,\frac12\right \},\{1,1,1,1,1\};-1\right \},\quad H_6=\left \{ \left \{\frac12,\frac12,\frac12\right \},\{1,1,1\};1  \right \},$$ respectively, both defined over $\Q$. We use the following in \texttt{Magma}.\medskip

\noindent \texttt{H5 := HypergeometricData([1/2,1/2,1/2,1/2,1/2], [1,1,1,1,1]);}

\noindent \texttt{Factorization(EulerFactor(H5,-1,5));} \medskip

We get \medskip

 \noindent   <25*\$.1 + 1, 1>,

 \noindent   <625*\$.$1^2$ - 34*\$.1 + 1, 1>,

 \noindent   <625*\$.$1^2$ + 30*\$.1 + 1, 1>

    \medskip

Meanwhile, using \medskip

\noindent \texttt{H6 := HypergeometricData([1/2,1/2,1/2], [1,1,1]);}

\noindent \texttt{Factorization(EulerFactor(H6,1,5));} \medskip

We get \medskip

<25*\$.$1^2$ + 6*\$.1 + 1, 1>

\medskip

 Further using the following we get a sequence denoted by $A_p$ where the subscript $p$ refers to the corresponding prime $p$. \medskip

\noindent \texttt{[Coefficient(EulerFactor(H5,-1,p),1)-p*Coefficient(EulerFactor(H6,1,p),1) }

\noindent\texttt{-LegendreSymbol(-3,p)*p$^2$: in PrimesUpTo(67) |p ge 7];}  \medskip

It produces the first few  $A_p$ where $p$ ranges from 7 to 67 listed as follows. \medskip

\noindent \texttt{[30,42,62,478,-200,128,400,-1922,-2338,2462,-8,4608,3600,5162,-6658,-6728]}

\medskip Note that they are not $p$th coefficients of GL(2) Hecke eigenforms as each odd weight Hecke eighenform with integer coefficients has to admit complex multiplication. By appearance, half of the $p$th coefficients of a CM modular form should vanish, which is not the case here. This case should be related to a GL(3) automorphic form, which is a symmetric square of a GL(2) automorphic form as conjectured by Beukers and Delaygue in \cite{BD}.  \medskip

Numerically we have for each prime $p\ge 7$

\begin{equation}
\pFq{5}{4}{\frac12&\frac12&\frac12&\frac12&\frac12}{&1&1&1&1}{-1}_{p-1}\equiv -A_p \mod p^2.
\end{equation} When $p$ is ordinary, it is already shown in \cite{BD} that the left hand side is congruent to the corresponding unit root modulo $p^2$.\medskip

\subsubsection{$\{\{\frac12,\frac12,\frac12,\frac13,\frac23\},\{1,1,1,\frac16,\frac56\};-1\}$}

Similarly, $a=b=c=\frac12$, $d=\frac 13,e=\frac23$ in Equation \eqref{eq:Whipple2}, we get another list $B_p$ using \medskip

\noindent \texttt{H7 := HypergeometricData([1/2,1/2,1/2,1/3,2/3], [1,1,1,1/6,5/6]);}

\noindent \texttt{H8 := HypergeometricData([1/2,1/3,2/3], [1,1,1]);}

\noindent \texttt{[Coefficient(EulerFactor(H7,-1,p),1)-p*Coefficient(EulerFactor(H8,1,p),1)}

\noindent \texttt{-LegendreSymbol(3,p)*p$^2$: in PrimesUpTo(67) | p ge 7];} \medskip

From which we get the first few values of $B_p$ listed as follows \medskip

\noindent \texttt{[34,-230,-290,542,588,-576,432,898,-2690,-994, 972, -2304, 8112,-5990,670,6348]}

\medskip

Numerically for each prime $p\ge 7$
\begin{equation}
p\cdot \, \pFq{5}{4}{\frac12&\frac12&\frac12&\frac13&\frac23}{&1&1&\frac76&\frac56}{-1}_{p-1}\equiv -B_p \mod p^2.
\end{equation}

\begin{remark}
Letting $a=b=c=d=\frac 12$ and $e=\frac 34$ in  Equation \eqref{eq:Whipple2}, one relates the hypergeometric datum $\{\{\frac12,\frac12,\frac12,\frac12\},\{1,1,1,1\};-1\}$ to $\{\{\frac 12,\frac12,\frac34\},\{1,1,1\};1\}$, which is not defined over $\Q$. See (5.15) of \cite{MP} by McCarthy and Papanikolas for the precise relation between the corresponding character sums when $p\equiv 1\mod 4$. Accordingly, we have the following numeric result. When $p\equiv 1\mod 4$
\begin{equation}
\pFq{4}{3}{\frac12&\frac12&\frac12&\frac12}{&1&1&1}{-1}_{p-1}\equiv \left(\frac{2}p\right)\frac{\G_p(\frac14)^2}{\G_p(\frac12)}a_p(f_{32.3.c.a}) \mod p^2.
\end{equation}

\end{remark}

\end{document}